\documentclass[letterpaper,10pt]{article}

\usepackage[margin=0.5in]{geometry}

\usepackage{fullpage}
\usepackage{enumerate}
\usepackage{authblk}
\usepackage[colorinlistoftodos]{todonotes}
\usepackage{verbatim}%comments out anything between \begin{comment} and \end{comment}
\usepackage{section, amsthm, textcase, setspace, amssymb, lineno, amsmath, amssymb, amsfonts, latexsym, fancyhdr, longtable, ulem, mathtools}
\usepackage{epsfig, graphicx, pstricks,pst-grad,pst-text,tikz,colortbl}
\usepackage{epsf}
\usepackage{graphicx, color}
\usepackage{float}
\usepackage[rflt]{floatflt}
\usepackage{amsfonts}
\usepackage{latexsym,enumitem}
\usetikzlibrary{fit,matrix,positioning}
\usepackage{pdflscape}
\usetikzlibrary{decorations.pathreplacing}
\usepackage{mathrsfs}
\usepackage{makecell}
\usetikzlibrary{decorations.markings}
\usepackage{tikz}
\usetikzlibrary{decorations.markings}
\usepackage{yfonts}
\usepackage[most]{tcolorbox}

\definecolor{darkgreen}{rgb}{0, 0.5, 0}

\newtheorem{theorem}{Theorem}
\newtheorem{lemma}{Lemma}
\newtheorem{corollary}{Corollary}

\newtheorem{Ex}{Example}

\newtheorem*{theorem*}{Theorem}
\newtheorem{remark}{Remark}

\newcommand{\ind}{{\rm ind \hspace{.1cm}}}

\newcommand{\C}{\mathbb{C}}

\newcommand{\g}{\mathfrak{g}}

\newcommand\blfootnote[1]{%
  \begingroup
  \renewcommand\thefootnote{}\footnote{#1}%
  \addtocounter{footnote}{-1}%
  \endgroup
}

\begin{document}

\title{Contact seaweeds }

\author[*]{Vincent E. Coll, Jr.}
\author[*]{Nicholas Mayers}
\author[*]{Nicholas Russoniello}
\author[$\dagger$] {Gil Salgado}

\affil[*]{Department of Mathematics, Lehigh University, Bethlehem, PA, 18015}
\affil[$\dagger$]{Facultad de Ciencias, Universidad Autonoma de San Luis Potos\'{i}; SLP, M\'{e}xico}

\maketitle
\begin{abstract}
\noindent
A ($2k+1$)$-$dimensional contact Lie algebra is one which admits a one-form 
$\varphi$ such that $\varphi \wedge (d\varphi)^k\ne0$.  Such algebras have index one, but this is not generally a sufficient condition. Here we show that index-one type-$A$ seaweed algebras are necessarily contact. Examples, together with a method for their explicit construction, are provided.
\end{abstract}



\blfootnote{Nicholas Mayers,  nwm5095@gmail.com, is the corresponding author.  Gil Salgado would like to acknowledge the support received through CONACYT Grant \#A1-S-45886 and
PRODEP grant UASLP-CA-228.
}

%\linenumbers

\section{Introduction}

The \textit{index} of a Lie algebra $(\g, [-,-])$ is an important algebraic invariant which was first formally introduced by Dixmier (\textbf{\cite{D}}, 1977). It is defined by 

\[
\ind \g=\min_{\varphi\in \mathfrak{g}^*} \dim  (\ker (B_\varphi)),
\]
where $\varphi$ is an element of the linear dual $\mathfrak{g}^*$  and $B_\varphi$ is the associated skew-symmetric \textit{Kirillov form} defined by 
\[
B_\varphi(x,y)=\varphi([x,y]), \textit{ for all }~ x,y\in\g. 
\]

\noindent
On a given Lie algebra $\g$, index-realizing linear forms, i.e., those $\varphi\in\mathfrak{g}^*$ which satisfy $\dim(\ker(B_{\varphi}))=\ind\mathfrak{g},$ are called \textit{regular} and exist in profusion, being dense in both the Zariski and Euclidean topolgies of $\g^*$ (see \textbf{\cite{DK}}).

\bigskip
\noindent
Using the above notation, the index is used to describe certain important classes of algebras.

\begin{itemize}

\item If dim $\mathfrak{g}=2n$ and if there exists a $\varphi$ such that $(d\varphi)^n\ne 0$, then $\mathfrak{g}$ is said to \textit{Frobenius}, and $\varphi$ is called a \textit{Frobenius form}. Deformation theorists are interested in Frobenius Lie algebras because each such $\mathfrak{g}$ provides a solution to the classical Yang-Baxter equation, which in turn quantizes to a universal deformation formula, i.e., a Drinfel'd twist which deforms any algebra which admits an action of $\mathfrak{g}$ by derivations (see \textbf{\cite{twist}}). A Lie algebra is Frobenius precisely when its index is zero.

\item If dim $\mathfrak{g}=2n+1$ and if there exists a $\varphi$ such that $\varphi  \wedge (d\varphi)^n\ne 0$, then $\mathfrak{g}$ is said to be $\textit{contact}$, $\varphi$ is called a \textit{contact form}, and $\varphi  \wedge (d\varphi)^n$ is a \textit{volume form} on the underlying Lie group. The construction and classification of contact manifolds is a central problem in differential topology (see \textbf{\cite{wein}}).
If a Lie algebra is contact, then its index is equal to one.  The converse is not true in general.\footnote{The Lie algebra $\g=\langle e_1,e_2,e_3\rangle$ with relations $[e_1,e_2]=e_2$ and $[e_1,e_3]=e_3$ has index one but is not contact.}  However,
there are a few important families of Lie algebras for which index one identifies if a given Lie algebra is contact. For example, index-one nilpotent Lie algebras and real, compact Lie algebras of index one are necessarily contact (see \textbf{\cite{RG}}).
\end{itemize}

\noindent
\begin{remark}\label{ex:ind1ncont}
Frobenius and contact Lie algebras are tightly interwoven.  Every Frobenius Lie algebra has a codimension-one contact ideal, and every Frobenius Lie algebra is a codimension-one ideal of a contact Lie algebra \textup(see \textup{\textbf{\cite{codim1}}}\textup).
\end{remark}

%\noindent
%\begin{Ex}\label{ex:notcontact}
%The Lie algebra $\g=\langle e_1,e_2,e_3\rangle$ with relations $[e_1,e_2]=e_2$ and $[e_1,e_3]=e_3$ has index one but is not contact.
%\end{Ex}

%\noindent
Here, we seek to classify contact algebras among a class of matrix algebras called \textit{seaweed algebras} (or simply, ``seaweeds'').  These algebras, along with their evocative name,  were first introduced by Dergachev and A. Kirillov in (\textbf{\cite{DK}}, 2000), where they defined such algebras as complex  subalgebras of $\mathfrak{gl}(n)$ preserving certain flags of subspaces 
of $\C^n$ developed from two compositions of $n$. The passage to seaweeds  of ``classical type" is realized by requiring that elements of the seaweed subalgebra of $\mathfrak{gl}(n)$ satisfy additional algebraic conditions. For example, the type-$A$ case ($A_{n-1}=\mathfrak{sl}(n,\mathbb{C})$) is defined by a vanishing trace condition.  Ongoing, we will assume that all Lie algebras are over $\mathbb{C}$.

\bigskip
We can now state the main result of this article, which asserts that index one is sufficient for seaweed subalgebras of $\mathfrak{sl}(n)=A_{n-1}$ to be contact.

\begin{theorem*}\label{main} If $\mathfrak{g}$ is a type-A seaweed, then $\mathfrak{g}$ is contact if and only if
$\ind\g=1.$
\end{theorem*}

The structure of the paper is as follows.  In Section~\ref{sec:seaweeds}, we define type-$A$ seaweeds and detail the construction of an associated planar graph, called a \textit{meander}, which is helpful in computing the seaweed's index. In Section~\ref{sec:framework}, we discuss a framework for constructing regular one-forms on type-$A$ seaweeds and explicitly compute the kernels of the associated Kirillov forms. In Section~\ref{sec:main}, we establish that the regular one-forms of Section~\ref{sec:framework}  are, in fact, contact, thus proving the main theorem. In Section~\ref{sec:examples}, the reader will find algebraic technology which can be used to generate a spate of both Frobenius and contact seaweeds of arbitrarily high dimension.

\section{Seaweeds and Meanders}\label{sec:seaweeds}
\noindent
A type-$A$ seaweed in its ``standard (matrix) form''\footnote{Since every seaweed is conjugate to one in standard form, we have presumed this in the definition of a type-$A$ seaweed in order to ease exposition.  A basis-free definition reckons seaweed subalgebras of a reductive Lie algebra $\mathfrak{g}$ as the intersection of two parabolic algebras whose sum is $\mathfrak{g}$ (see \textbf{\cite{Pan}}).  For this reason, seaweed algebras have elsewhere been called \textit{biparabolic} (see \textbf{\cite{Joseph}}).  We do not require the latter definition for our present discussion.} can be described as a subalgebra of $\mathfrak{sl}(n)$ constructed as follows.  First, fix two ordered compositions of $n$, $(a)=(a_1,\dots,a_m)$ and $(b)=(b_1,\dots,b_t)$. Let $D_{(a)}$ be the subalgebra of block-diagonal matrices whose blocks have sizes $a_1\times a_1,\dots,a_m\times a_m$ and similarly for $D_{(b)}$.  A \textit{type-A seaweed algebra} (or simply, ``type-A seaweed")  is the subalgebra of $\mathfrak{sl}(n)$ spanned by the intersection of $D_{(a)}$ with the lower-triangular matrices and the intersection of $D_{(b)}$ with the upper-triangular matrices. We call the the locations of potentially nonzero entries in the seaweed \textit{admissible locations}. For a type-$A$ seaweed defined by two compositions of $n$ as above, we write $\mathfrak{p}_n^A\;\frac{a_1|\cdots|a_m}{b_1|\cdots|b_t}$. See Example~\ref{ex:seaweed}.

%\footnote{A basis-free definition reckons seaweed subalgebras of a simple Lie algebra $\mathfrak{g}$ as the intersection of two parabolic algebras whose sum is $\mathfrak{g}$ (see \textbf{[16]}).  For this reason, seaweeds have elsewhere been called \textit{biparabolic} (see \textbf{[13]}).  We do not require the latter definition for our present discussion nor are we here concerned with seaweeds in the other classical, or exceptional, types.}

\begin{remark} We tacitly assume a standard \textup(Chevalley\textup) basis for a seaweed $\g=\mathfrak{p}_n^A \frac{a_1|\dots|a_m}{b_1|\dots|b_t}$ given by the union of the following two sets of matrix units:

\begin{itemize}
    \item $\{e_{i,i}-e_{i+1,i+1} ~|~  1\leq i\leq n-1   \}$,  and
    \item $\{e_{i,j}~|~ 1\leq i\neq j\leq n$ and $(i,j)$ is an admissible location\}.
\end{itemize}
\end{remark}

\begin{Ex}\label{ex:seaweed}
We illustrate a type-$A$ seaweed in its standard matrix form -- revealing its characteristic wavy seaweed ``shape.'' The asterisks represent admissible locations and entries in non-admissible locations are zeroes.
See Figure \ref{Aseaweed}.
\end{Ex}

%Let $(a_1,\cdots,a_m)$ and $(b_1,\cdots,b_t)$ be two ordered partitions of $n$, and let $\{0\}=V_0\subset V_1\subset\cdots\subset V_m=\C^n$, and $\C^n=W_0\supset W_1\supset\cdots\supset W_t=\{0\},$ where $V_i=\text{span} \{e_1,\cdots,e_{a_1+\cdots+a_i}\}$ and $W_j=\text{span}\{e_{b_1+\cdots+b_j+1},\cdots,e_n\}$. The {standard} {seaweed} $\mathfrak{g}$ of \textit{{type}} $\frac{a_1|\cdots|a_m}{b_1|\cdots|b_t}$ is the subalgebra of $\mathfrak{gl}(n)$ which preserves the spaces $V_i$ and $W_j$.
% If the seaweed of type $\frac{a_1|\cdots|a_m}{b_1|\cdots|b_t}$ is further required to have trace zero, then this seaweed is said to be of Type-$A$ and is denoted 

%Any seaweed is conjugate, over its algebraic group, to a 
%seaweed in this \textit{standard} form and one then has the standard matrix representation. See Figure 1 (left) where the seaweed can be visualized inside a matrix.  The *'s indicate values of $\C$.  All other entries are zero. 

\begin{figure}[H]
$$\begin{tikzpicture}[scale=.6]
\draw (0,0) -- (0,6);
\draw (0,6) -- (6,6);
\draw (6,6) -- (6,0);
\draw (6,0) -- (0,0);
\draw [line width=3](0,6) -- (0,4);
\draw [line width=3](0,4) -- (2,4);
\draw [line width=3](2,4) -- (2,0);
\draw [line width=3](2,0) -- (6,0);

\draw [line width=3](0,6) -- (1,6);
\draw [line width=3](1,6) -- (1,5);
\draw [line width=3](1,5) -- (3,5);
\draw [line width=3](3,5) -- (3,3);
\draw [line width=3](3,3) -- (6,3);
\draw [line width=3](6,3) -- (6,0);

\draw [dotted] (0,6) -- (6,0);

\node at (.5,5.4) {{\LARGE *}};
\node at (.5,4.4) {{\LARGE *}};
\node at (1.5,4.4) {{\LARGE *}};
\node at (2.5,4.4) {{\LARGE *}};
\node at (2.5,3.4) {{\LARGE *}};
\node at (2.5,2.4) {{\LARGE *}};
\node at (3.5,2.4) {{\LARGE *}};
\node at (4.5,2.4) {{\LARGE *}};
\node at (5.5,2.4) {{\LARGE *}};
\node at (2.5,1.4) {{\LARGE *}};
\node at (3.5,1.4) {{\LARGE *}};
\node at (4.5,1.4) {{\LARGE *}};
\node at (5.5,1.4) {{\LARGE *}};
\node at (2.5,0.4) {{\LARGE *}};
\node at (3.5,0.4) {{\LARGE *}};
\node at (4.5,0.4) {{\LARGE *}};
\node at (5.5,0.4) {{\LARGE *}};

\node[label=above:{1}] at (0.5,6) {};
\node[label=above:{2}] at (2,5) {};
\node[label=above:{3}] at (4.5,3) {};
\node[label=left:{2}] at (0,5) {};
\node[label=left:{4}] at (2,2) {};

\end{tikzpicture}$$

\caption{
The seaweed $\mathfrak{p}_6^A\frac{2|4}{1|2|3}$}
\label{Aseaweed}
\end{figure}
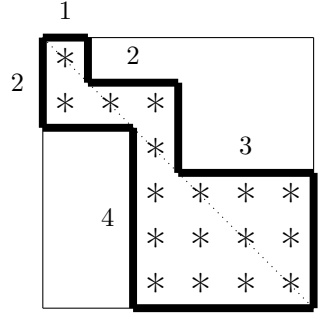

To each seaweed $\mathfrak{p}_n^A\;\frac{a_1|\cdots|a_m}{b_1|\cdots|b_t}$ we associate a planar graph called a \textit{{meander}}, constructed as follows.  First, place $n$ vertices $v_1$ through $v_n$ in a horizontal line.  Next, create two partitions of the vertices by forming \textit{{top}} and {\textit{bottom blocks}} of vertices of size $a_1$, $a_2$, $\cdots$, $a_m$, and $b_1$, $b_2$, $\cdots$, $b_t$, respectively.  Place edges in each top (likewise bottom) block in the same way. Add an edge from the first vertex of the block to the last vertex of the same block.  Repeat this edge addition on the second vertex and the second to last vertex within the same block and so on within each block of both partitions. Top edges are drawn concave down and bottom edges are drawn concave up. Let $M(\g)$ denote the meander associated with $\g$. We place a counterclockwise orientation on $M(\g)$ to produce the \textit{directed meander} $\overrightarrow{M}(\g).$ See Example~\ref{ex:meanders}.

%Let $e_{i,j}$ be the matrix which has an entry of 1 in the $(i,j)$th position, and 0 in all other entries. A meander can be visualized inside its associated seaweed $\mathfrak{g}$ if one views the diagonal entries $e_{i,i}$ of $\mathfrak{g}$ as the $n$ vertices $v_i$ of the meander and reckons the top edges $(v_i,v_j)$ with $i<j$ of the meander as the unions of line segments connecting the matrix locations $(i,i)\rightarrow(j,i)\rightarrow(j,j)$ and the bottom edges $(v_i,v_j)$ with $i<j$ of the meander as the unions of line segments connecting the matrix locations $(i,i)\rightarrow(i,j)\rightarrow(j,j)$. See Figure \ref{MeanderinSeaweed} \textup(right\textup).

\begin{Ex}\label{ex:meanders}  We illustrate the meander and directed meander of $\g=\mathfrak{p}_8^A\frac{2|6}{8}.$ 
See Figure 2.

\begin{figure}[H]
$$\begin{tikzpicture}[scale=0.8]
\def\Node{\node [circle, fill, inner sep=1.5pt]}
\tikzset{->-/.style={decoration={
  markings,
  mark=at position .55 with {\arrow{>}}},postaction={decorate}}}
  \Node[label=below:\footnotesize {$v_1$}] at (0,0) {};
  \Node[label=below:\footnotesize {$v_2$}] at (1,0) {};
  \Node[label=below:\footnotesize {$v_3$}] at (2,0) {};
  \Node[label=below:\footnotesize {$v_4$}] at (3,0) {};
  \Node[label=below:\footnotesize {$v_5$}] at (4,0) {};
  \Node[label=below:\footnotesize {$v_6$}] at (5,0){};
  \Node[label=below:\footnotesize{$v_7$}] at (6,0){};
  \Node[label=below:\footnotesize{$v_8$}] at (7,0){};
  \draw (0,0) to[bend left=60] (1,0);
  \draw (6,0) to[bend left=60] (1,0);
  \draw (5,0) to[bend left=60] (2,0);
  \draw (4,0) to[bend left=60] (3,0);
  \draw (2,0) to[bend left=60] (7,0);
  \draw (3,0) to[bend left=60] (6,0);
  \draw (4,0) to[bend left=60] (5,0);
  \draw (7,0) to[bend left=60] (0,0);
\end{tikzpicture}
\hspace{5em}
\begin{tikzpicture}[scale=0.8]
\def\Node{\node [circle, fill, inner sep=1.5pt]}
\tikzset{->-/.style={decoration={
  markings,
  mark=at position .55 with {\arrow{>}}},postaction={decorate}}}
  \Node[label=below:\footnotesize {$v_1$}] at (0,0) {};
  \Node[label=below:\footnotesize {$v_2$}] at (1,0) {};
  \Node[label=below:\footnotesize {$v_3$}] at (2,0) {};
  \Node[label=below:\footnotesize {$v_4$}] at (3,0) {};
  \Node[label=below:\footnotesize {$v_5$}] at (4,0) {};
  \Node[label=below:\footnotesize {$v_6$}] at (5,0){};
  \Node[label=below:\footnotesize{$v_7$}] at (6,0){};
  \Node[label=below:\footnotesize{$v_8$}] at (7,0){};
  \draw[->-] (1,0) to[bend right=60] (0,0);
  \draw[->-] (1,0) to[bend right=60] (6,0);
  \draw[->-] (2,0) to[bend right=60] (5,0);
  \draw[->-] (3,0) to[bend right=60] (4,0);
  \draw[->-] (7,0) to[bend right=60] (2,0);
  \draw[->-] (6,0) to[bend right=60] (3,0);
  \draw[->-] (5,0) to[bend right=60] (4,0);
  \draw[->-] (0,0) to[bend right=60] (7,0);
\end{tikzpicture}$$

\vspace{-1em}

\caption{$M(\g)$ and $\protect\overrightarrow{M}(\g)$}
\label{MeanderinSeaweed}
\end{figure}

\end{Ex}

%\begin{Ex} 
%\label{MeanderEx}
%Consider $\mathfrak{g}$ of type $\frac{4|1}{2|1|2}$. The %meander $M(\mathfrak{g})$ is illustrated in Figure \%ref{MeanderinSeaweed} \textup(left\textup).
%\end{Ex}

\noindent A meander can be visualized inside its associated seaweed $\g$ if one views the diagonal locations $\{(i,i)\}_{i=1}^n$ of $\g$ as the $n$ vertices $\{v_i\}_{i=1}^n$ of the meander and reckons the top edges $\{(v_i,v_j) ~|~  i<j\}$  of the meander as the unions of line segments connecting the matrix locations $(i,i)\rightarrow(j,i)\rightarrow(j,j)$ and the bottom edges $\{(v_i,v_j) ~|~ i<j\}$ of the meander as the unions of line segments connecting the matrix locations $(i,i)\rightarrow(i,j)\rightarrow(j,j)$. See Figure \ref{fig:MeanderinSeaweed}.

\begin{figure}[H]
$$\begin{tikzpicture}[scale=0.45]
\def\Node{\node [circle, fill, inner sep=3pt]}
	\draw (0,0) -- (0,8);
	\draw (0,8) -- (8,8);
	\draw (8,8) -- (8,0);
	\draw (8,0) -- (0,0);
	\draw [line width=3](0,8) -- (0,6);
	\draw [line width=3](0,6) -- (2,6);
	\draw [line width=3](2,6) -- (2,0);
	\draw [line width=3](2,0) -- (8,0);
	\draw [line width=3](0,8) -- (8,8);
	\draw [line width=3](8,8) -- (8,0);
	\draw[dotted] (0,8)--(8,0);
	\node at (0.5,7.4) {{\LARGE *}};
	\node at (1.5,7.4) {{\LARGE *}};
	\node at (2.5,7.4) {{\LARGE *}};
	\node at (3.5,7.4) {{\LARGE *}};
	\node at (4.5,7.4) {{\LARGE *}};
	\node at (5.5,7.4) {{\LARGE *}};
	\node at (6.5,7.4) {{\LARGE *}};
	\node at (7.5,7.4) {{\LARGE *}};
	\node at (0.5,6.4) {{\LARGE *}};
	\node at (1.5,6.4) {{\LARGE *}};
	\node at (2.5,6.4) {{\LARGE *}};
	\node at (3.5,6.4) {{\LARGE *}};
	\node at (4.5,6.4) {{\LARGE *}};
	\node at (5.5,6.4) {{\LARGE *}};
	\node at (6.5,6.4) {{\LARGE *}};
	\node at (7.5,6.4) {{\LARGE *}};
	\node at (2.5,5.4) {{\LARGE *}};
	\node at (3.5,5.4) {{\LARGE *}};
	\node at (4.5,5.4) {{\LARGE *}};
	\node at (5.5,5.4) {{\LARGE *}};
	\node at (6.5,5.4) {{\LARGE *}};
	\node at (7.5,5.4) {{\LARGE *}};
	\node at (2.5,4.4) {{\LARGE *}};
	\node at (3.5,4.4) {{\LARGE *}};
	\node at (4.5,4.4) {{\LARGE *}};
	\node at (5.5,4.4) {{\LARGE *}};
	\node at (6.5,4.4) {{\LARGE *}};
	\node at (7.5,4.4) {{\LARGE *}};
	\node at (2.5,3.4) {{\LARGE *}};
	\node at (3.5,3.4) {{\LARGE *}};
	\node at (4.5,3.4) {{\LARGE *}};
	\node at (5.5,3.4) {{\LARGE *}};
	\node at (6.5,3.4) {{\LARGE *}};
	\node at (7.5,3.4) {{\LARGE *}};
	\node at (2.5,2.4) {{\LARGE *}};
	\node at (3.5,2.4) {{\LARGE *}};
	\node at (4.5,2.4) {{\LARGE *}};
	\node at (5.5,2.4) {{\LARGE *}};
	\node at (6.5,2.4) {{\LARGE *}};
	\node at (7.5,2.4) {{\LARGE *}};
	\node at (2.5,1.4) {{\LARGE *}};
	\node at (3.5,1.4) {{\LARGE *}};
	\node at (4.5,1.4) {{\LARGE *}};
	\node at (5.5,1.4) {{\LARGE *}};
	\node at (6.5,1.4) {{\LARGE *}};
	\node at (7.5,1.4) {{\LARGE *}};
	\node at (2.5,0.4) {{\LARGE *}};
	\node at (3.5,0.4) {{\LARGE *}};
	\node at (4.5,0.4) {{\LARGE *}};
	\node at (5.5,0.4) {{\LARGE *}};
	\node at (6.5,0.4) {{\LARGE *}};
	\node at (7.5,0.4) {{\LARGE *}};
	\node [label=above:{8}] at (4,8){};
	\node [label=left:{2}] at (0,7) {};
	\node [label=left:{6}] at (2,3) {};
	
	\draw[line width=1.5] (0.5,7.6)--(7.5,7.6)--(7.5,0.6)--(2.5,0.6)--(2.5,5.6)--(5.5,5.6)--(5.5,2.6)--(4.5,2.6)--(4.5,4.6)--(3.5,4.6)--(3.5,1.6)--(6.5,1.6)--(6.5,6.6)--(0.5,6.6)--(0.5,7.6);
\end{tikzpicture}$$
    \caption{The meander of $\mathfrak{g}=\mathfrak{p}_8^A\frac{2|6}{8}$ visualized in $\mathfrak{g}$}
    \label{fig:MeanderinSeaweed}
\end{figure}

Remarkably, the index of the seaweed can be computed by counting the number and type of the connected components in the associated meander. In a given meander, we call a connected component a \textit{path} if it is not a cycle.  Note that a path may be degenerate, i.e., consist of a single vertex.

\begingroup
\makeatletter
\apptocmd{\thetheorem}{  \unless\ifx\protect\@unexpandable@protect\textnormal{(Dergachev and A. Kirillov} \textbf{\cite{DK}}, \textnormal{2000)}\protect\footnote{The authors actually established the formula $2C+P$ for the index of a seaweed subalgebra of $\mathfrak{gl}(n)$, but only a minor algebraic argument is required to extend to the type-$A$ case yielding (\ref{eqn:indexA}). See \textbf{\cite{seriesA}}.}\fi}{}{}
\makeatother

\begin{theorem}\label{thm:indexA} 
If $\mathfrak{g}$ is a type-A seaweed, 
%$\mathfrak{p}_n^A\frac{a_1|\cdots|a_m}{b_1|\cdots|b_t}$, 
then
\begin{equation}\label{eqn:indexA}
\ind\mathfrak{g}=2C+P-1,
\end{equation}
\noindent 
 where $C$ is the number of cycles and $P$ is the number of paths in  $M(\mathfrak{g})$.
\end{theorem}
\endgroup

%\begin{Ex}
%The type-A seaweed $\mathfrak{p}_8^A\frac{2|6}{8},$ whose meander is shown in Figure~\ref{MeanderinSeaweed}, has index equal to $1,$ since its meander consists of exactly one cycle.
%\end{Ex}

\begin{Ex}
Consider the type-A seaweed $\mathfrak{b}=\mathfrak{p}_5^A\frac{1|1|1|1|1}{5},$ which is the standard Borel subalgebra of $\mathfrak{sl}(5)$ \textup(see Figure~\ref{fig:borel} \textup(left\textup)\textup).  Since $M(\mathfrak{b})$ consists of zero cycles and three paths  \textup(see Figure~\ref{fig:borel} \textup(right\textup)\textup), it follows from Theorem~\ref{thm:indexA} that $$\ind\mathfrak{b}=2(0)+3-1=2.$$ 

\begin{figure}[H]
$$\begin{tikzpicture}[scale=.6]
\draw (0,0)--(5,0)--(5,5)--(0,5)--(0,0);

\draw[dotted] (5,0)--(0,5);

\draw[line width=3] (0,5)--(5,5);
\draw[line width=3] (5,5)--(5,0);

\draw[line width=3] (0,5)--(0,4);
\draw[line width=3] (0,4)--(1,4);
\draw[line width=3] (1,4)--(1,3);
\draw[line width=3] (1,3)--(2,3);
\draw[line width=3] (2,3)--(2,2);
\draw[line width=3] (2,2)--(3,2);
\draw[line width=3] (3,2)--(3,1);
\draw[line width=3] (3,1)--(4,1);
\draw[line width=3] (4,1)--(4,0);
\draw[line width=3] (4,0)--(5,0);

\node at (0.5,4.4){{\LARGE *}};
\node at (1.5,4.4){{\LARGE *}};
\node at (2.5,4.4){{\LARGE *}};
\node at (3.5,4.4){{\LARGE *}};
\node at (4.5,4.4){{\LARGE *}};
\node at (1.5,3.4){{\LARGE *}};
\node at (2.5,3.4){{\LARGE *}};
\node at (3.5,3.4){{\LARGE *}};
\node at (4.5,3.4){{\LARGE *}};
\node at (2.5,2.4){{\LARGE *}};
\node at (3.5,2.4){{\LARGE *}};
\node at (4.5,2.4){{\LARGE *}};
\node at (3.5,1.4){{\LARGE *}};
\node at (4.5,1.4){{\LARGE *}};
\node at (4.5,.4){{\LARGE *}};

\node[label=above:{5}] at (2.5,5){};
\node[label=left:{1}] at (0,4.5){};
\node[label=left:{1}] at (1,3.5){};
\node[label=left:{1}] at (2,2.5){};
\node[label=left:{1}] at (3,1.5){};
\node[label=left:{1}] at (4,.5){};
\end{tikzpicture}
\hspace{5em}
\begin{tikzpicture}
\def\Node{\node [circle, fill, inner sep=1.5pt]}

\Node[label=above:{$v_1$}] (1) at (0,0){};
\Node[label=above:{$v_2$}] (2) at (1,0){};
\Node[label=above:{$v_3$}] (3) at (2,0){};
\Node[label=above:{$v_4$}] (4) at (3,0){};
\Node[label=above:{$v_5$}] (5) at (4,0){};

\draw (1) to[bend right=60] (5);
\draw (2) to[bend right=60] (4);
\end{tikzpicture}$$
    \caption{The Borel $\mathfrak{b}$ and $M(\mathfrak{b}) $}
    \label{fig:borel}
\end{figure}
\end{Ex}

Any meander can be contracted, or ``wound down," to the empty meander through a sequence of graph-theoretic moves -- detailed in Lemma~\ref{lem:winding} below -- each of which is uniquely determined by the structure of the meander at the time of the move application. Such a sequence is called the \textit{signature} of the meander (see \textbf{ \cite{meanders2}}).  The signature is essentially a graph theoretic recasting of Panyushev’s reduction algorithm  (see \textbf{\cite{Pan}}).
\bigskip

\begin{tcolorbox}[breakable, enhanced]
\begin{lemma}[Coll et al. \textbf{\cite{meanders2}}]\label{lem:winding}
\label{winding down}
Let $\g=\mathfrak{p}_n^A\frac{a_1|\dots|a_m}{b_1|\dots|b_t}$ be a type-A seaweed with associated meander $M(\g)$ -- in this setting, we say the meander $M=M(\mathfrak{g})$ is of type $\frac{a_1|\dots|a_m}{b_1|\dots|b_t}.$ Create a meander $M'$ by one of the following moves.
\begin{enumerate}
    \item {Pure Contraction \textup(P\textup):} If $a_1>2b_1$, then $M\mapsto M'$ of type $\frac{(a_1-2b_1)|b_1|a_2|\cdots|a_m}{b_2|b_3|\cdots|b_t}$.
	\item {Block Elimination  \textup(B\textup):} If $a_1=2b_1$, then $M\mapsto M'$ of type $\frac{b_1|a_2|\cdots|a_m}{b_2|b_3|\cdots|b_t}$.
    \item {Rotation Contraction \textup(R\textup):} If $b_1<a_1<2b_1$, then $M\mapsto M'$ of type $\frac{b_1|a_2|\cdots|a_m}{(2b_1-a_1)|b_2|\cdots|b_t}$.
    \item {Component Deletion \textup(C\textup(c\textup)\textup):} If $a_1=b_1=c$, then $M\mapsto M'$ of type $\frac{a_2|\cdots|a_m}{b_2|\cdots|b_t}$.
    \item {Flip \textup(F\textup):} If $a_1<b_1$, then $M\mapsto M'$ of type $\frac{b_1|b_2|\cdots|b_t}{a_1|\cdots|a_m}$.
\end{enumerate}
  These moves are called \textit{{winding-down moves}}.  For all moves, except the Component Deletion move, $\g$ and $\g'$ \textup(the seaweed with meander $M(\g')=M'$\textup) have the same index. 
\end{lemma}
\end{tcolorbox}
\bigskip

\noindent
If $\mathfrak{g}$ is as in Lemma \ref{winding down} and $M(\mathfrak{g})$ is a meander for which the collection of component deletions in its signature is $\{C(c_1),\dots,C(c_q)\}$, then we say that $M(\mathfrak{g})$ and,  by an abuse of terminology,  $\mathfrak{g}$ each have \textit{homotopy type} $\mathcal{H}(c_1,\dots,c_q).$  See Example \ref{ex:winding}.

\begin{Ex}\label{ex:winding}
The meander of type $\frac{2|6}{8}$ of Example~\ref{ex:meanders} has signature $FPFRBC(2),$ and so has homotopy type $\mathcal{H}(2).$ The winding-down moves associated with the signature are illustrated in Figure~\ref{fig:winding}.

\begin{figure}[H]
$$\begin{tikzpicture}
[scale=.34]
\def\Node{\node [circle, fill, inner sep=1.5pt]}
\node at (3.5,-2.75){$\frac{2|6}{8}$};
\Node(1) at (0,0){};
\Node(2) at (1,0){};
\Node(3) at (2,0){};
\Node(4) at (3,0){};
\Node(5) at (4,0){};
\Node(6) at (5,0){};
\Node(7) at (6,0){};
\Node(8) at (7,0){};
\draw (1) to[bend right=60](8);
\draw (1) to[bend left=60](2);
\draw (2) to[bend right=60](7);
\draw (3) to[bend right=60](6);
\draw (4) to[bend right=60](5);
\draw (3) to[bend left=60](8);
\draw (4) to[bend left=60](7);
\draw (5) to[bend left=60](6);
\end{tikzpicture}
\hspace{0.5em}
\begin{tikzpicture}[scale=.34]
\def\Node{\node [circle, fill, inner sep=1.5pt]}
\node at (-1.5,0){$\overset{F}{\mapsto}$};
\node at (3.5,-2.75){$\frac{8}{2|6}$};
\Node(1) at (0,0){};
\Node(2) at (1,0){};
\Node(3) at (2,0){};
\Node(4) at (3,0){};
\Node(5) at (4,0){};
\Node(6) at (5,0){};
\Node(7) at (6,0){};
\Node(8) at (7,0){};
\draw (1) to[bend left=60](8);
\draw (1) to[bend right=60](2);
\draw (2) to[bend left=60](7);
\draw (3) to[bend left=60](6);
\draw (4) to[bend left=60](5);
\draw (3) to[bend right=60](8);
\draw (4) to[bend right=60](7);
\draw (5) to[bend right=60](6);
\end{tikzpicture}
\hspace{0.5em}
\begin{tikzpicture}[scale =.34]
\def\Node{\node [circle, fill, inner sep=1.5pt]}
\node at (-1.5,0){$\overset{P}{\mapsto}$};
\node at (2.5,-2.75){$\frac{4|2}{6}$};
\Node at (0,0){};
\Node at (1,0){};
\Node at (2,0){};
\Node at (3,0){};
\Node at (4,0){};
\Node at (5,0){};
\draw (0,0) to[bend left=60] (3,0);
\draw (1,0) to[bend left=60] (2,0);
\draw (4,0) to[bend left=60] (5,0);
\draw (0,0) to[bend right=60] (5,0);
\draw (1,0) to[bend right=60] (4,0);
\draw (2,0) to[bend right=60] (3,0);
\end{tikzpicture}
\hspace{0.5em}
\begin{tikzpicture}[scale =.34]
\def\Node{\node [circle, fill, inner sep=1.5pt]}
\node at (-1.5,0){$\overset{F}{\mapsto}$};
\node at (2.5,-2.75){$\frac{6}{4|2}$};
\Node at (0,0){};
\Node at (1,0){};
\Node at (2,0){};
\Node at (3,0){};
\Node at (4,0){};
\Node at (5,0){};
\draw (0,0) to[bend right=60] (3,0);
\draw (1,0) to[bend right=60] (2,0);
\draw (4,0) to[bend right=60] (5,0);
\draw (0,0) to[bend left=60] (5,0);
\draw (1,0) to[bend left=60] (4,0);
\draw (2,0) to[bend left=60] (3,0);
\end{tikzpicture}
\hspace{0.5em}
\begin{tikzpicture}[scale=.34]
\def\Node{\node [circle, fill, inner sep=1.5pt]}
\node at (-1.5,0){$\overset{R}{\mapsto}$};
\node at (1.5,-2.75){$\frac{4}{2|2}$};
\Node at (0,0){};
\Node at (1,0){};
\Node at (2,0){};
\Node at (3,0){};
\draw (0,0) to[bend left=60] (3,0);
\draw (1,0) to[bend left=60] (2,0);
\draw (0,0) to[bend right=60] (1,0);
\draw (2,0) to[bend right=60] (3,0);
\end{tikzpicture}
\hspace{0.5em}
\begin{tikzpicture}[scale=.34]
\def\Node{\node [circle, fill, inner sep=1.5pt]}
\node at (-1.5,0){$\overset{B}{\mapsto}$};
\node at (0.5,-2.75){$\frac{2}{2}$};
\Node at (0,0){};
\Node at (1,0){};
\draw (0,0) to[bend left=90] (1,0);
\draw (0,0) to[bend right=90] (1,0);
\end{tikzpicture}
\hspace{0.5em}
\begin{tikzpicture}[scale=.34]
\node at (-1.5,-1){$\overset{C(2)}{\mapsto}$};
\node at (0,-1.2){$\emptyset$};
\node at (2,-3.75){\color{white}$\frac{0}{0}$};
\end{tikzpicture}$$
    \caption{Winding down the meander of type $\frac{2|6}{8}$}
    \label{fig:winding}
\end{figure}
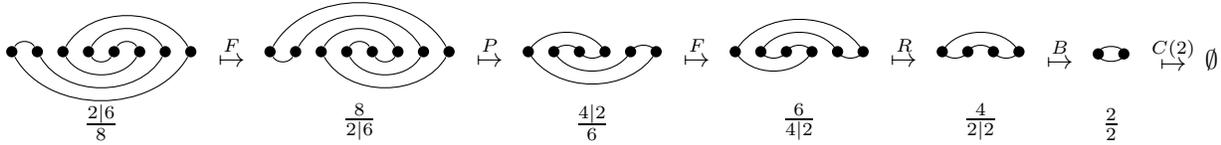
\end{Ex}

Since contact seaweeds are the main focus, we need only consider those seaweeds with index one. Theorem $\ref{thm:indexA}$ tells us how to find them. For emphasis, we record this in the following corollary to Theorem~\ref{thm:indexA}.

\begin{theorem}\label{cor:index A}
A type-A seaweed has index one if and only if its associated meander consists of exactly one cycle or exactly two paths, i.e., its associated meander has homotopy type $\mathcal{H}(2)$ or $\mathcal{H}(1,1).$
\end{theorem}

\section{Framework for regular forms}\label{sec:framework}

In her dissertation, Dougherty (\textbf{\cite{Adiss}}, 2019) establishes a complicated inductive framework for the explicit construction of families of regular one-forms on seaweeds of classical type (cf. \textbf{\cite{aria}}). When 
restricted to type-$A$ seaweeds of index one, this framework simplifies considerably. Importantly, some of these regular families consist of one-forms that are also contact.
We detail this ``type-$A$ framework'' in the following subsections.

\subsection{Component meanders}
\textit{Notation}:  Ongoing, $\mathfrak{g}$ will be assumed to be a type-$A$ seaweed.

\medskip
If $\mathfrak{g}$ has homotopy type $\mathcal{H}(c_1,\dots,c_q),$ define the
\textit{component meander} $CM(\g)$ associated with $\g$ to be the meander with the same signature as $\g$ except that the component deletions are all of size one. The construction of $CM(\mathfrak{g})$ involves an implicit identification of components, i.e., vertices and edges, in $M(\mathfrak{g})$. The vertices of $CM(\mathfrak{g})$ are $\{v_{I_1},\dots,v_{I_t}\},$ where $I_i$ is the set of indices for the vertices in $M(\mathfrak{g})$ that were collapsed into one vertex in $CM(\mathfrak{g}).$ Note that $|I_i|=c_j$ if $v_{I_i}$ is a vertex of $CM(\mathfrak{g})$ arising from the collapsing of a component of size $c_j.$ Also note $CM(\mathfrak{g})$ may be oriented as usual to yield $\overrightarrow{CM}(\mathfrak{g})$.  See Example \ref{ex:componentmeander}.

\begin{Ex}\label{ex:componentmeander}
Consider $\mathfrak{g}=\mathfrak{p}_6^A\frac{2|1|1|2}{6}.$
In Figure \ref{fig:commeanex}, we illustrate the meander, component meander, and directed component meander of $\mathfrak{g}$.

%We illustrate the definitions above by constructing the component meander, core set, and peak set of the seaweed of Example~\ref{ex:decomp}, $\g=\mathfrak{p}_6^A\frac{2|1|1|2}{6}.$ Recall $\mathfrak{g}$ has homotopy type $\mathcal{H}(2,1).$

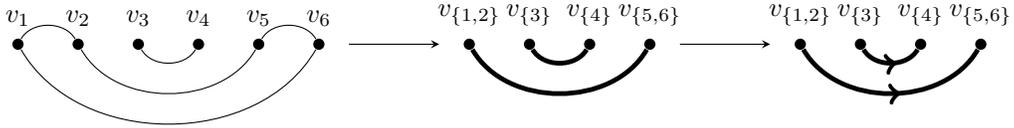
\begin{figure}[H]
$$\begin{tikzpicture}[scale=0.8]
\def\Node{\node [circle, fill, inner sep=1.5pt]}
\tikzset{->-/.style={decoration={
  markings,
  mark=at position .55 with {\arrow{>}}},postaction={decorate}}}
\tikzset{-->-/.style={decoration={
  markings,
  mark=at position .6 with {\arrow{>}}},postaction={decorate}}}
\Node (1) at (0,0){};
\Node (2) at (1,0){};
\Node (3) at (2,0){};
\Node (4) at (3,0){};
\Node (5) at (4,0){};
\Node (6) at (5,0){};
\draw (1) to[bend right=60] (6);
\draw (2) to[bend right=60] (5);
\draw (3) to[bend right=60] (4);
\draw (1) to[bend left=60] (2);
\draw (5) to[bend left=60] (6);

\node[label=above:{$v_1$}] at (0,0){};
\node[label=above:{$v_2$}] at (1,0){};
\node[label=above:{$v_3$}] at (2,0){};
\node[label=above:{$v_4$}] at (3,0){};
\node[label=above:{$v_5$}] at (4,0){};
\node[label=above:{$v_6$}] at (5,0){};

\draw[-stealth] (5.5,0)--(7,0);

\Node (7) at (7.5,0){};
\Node (8) at (8.5,0){};
\Node (9) at (9.5,0){};
\Node (10) at (10.5,0){};
\draw[line width=1.8] (7) to[bend right=60] (10);
\draw[line width=1.8] (8) to[bend right=60] (9);

\node[label=above:{$v_{\{1,2\}}$}] at (7.5,0){};
\node[label=above:{$v_{\{3\}}$}] at (8.5,0){};
\node[label=above:{$v_{\{4\}}$}] at (9.5,0){};
\node[label=above:{$v_{\{5,6\}}$}] at (10.5,0){};

\draw[-stealth] (11,0)--(12.5,0);

\Node (11) at (13,0){};
\Node (12) at (14,0){};
\Node (13) at (15,0){};
\Node (14) at (16,0){};
\draw[line width=1.8,->-] (11) to[bend right=60] (14);
\draw[line width=1.8,-->-] (12) to[bend right=60] (13);

\node[label=above:{$v_{\{1,2\}}$}] at (13,0){};
\node[label=above:{$v_{\{3\}}$}] at (14,0){};
\node[label=above:{$v_{\{4\}}$}] at (15,0){};
\node[label=above:{$v_{\{5,6\}}$}] at (16,0){};
\end{tikzpicture}$$
\caption{$M(\mathfrak{g})$, $CM(\mathfrak{g})$, and $\protect\overrightarrow{CM}(\mathfrak{g})$}
\label{fig:commeanex}
\end{figure}

\end{Ex}

\subsection{The Core and Peak}

The core and peak  of $\mathfrak{g}$ are sets of sets of admissible locations and their definitions are based on a vector space decomposition of $\mathfrak{g}$ developed from the homotopy type of $\mathfrak{g}$ as follows. If  $\mathfrak{g}$ has homotopy type $\mathcal{H}(c_1,\dots,c_q)$ then 
\begin{equation}\label{eqn:decomp}
\mathfrak{g}=\bigoplus_{i=1}^q\mathfrak{g}|_{c_i},
\end{equation} where $\mathfrak{g}|_{c_i}$ is the subspace of $\mathfrak{g}$ corresponding to a particular component of size $c_i$ in $M(\mathfrak{g}).$ See Example~\ref{ex:decomp}.

\begin{Ex}\label{ex:decomp}
Consider $\mathfrak{g}=\mathfrak{p}_6^A\frac{2|1|1|2}{6}$ of our running Example \ref{ex:componentmeander}. Note that $\mathfrak{g}$ has homotopy type $\mathcal{H}(2,1)$, which yields the vector space decomposition $\mathfrak{g}=\mathfrak{g}|_2\oplus\mathfrak{g}|_1.$ See Figure~\ref{fig:decomp}.

\begin{figure}[H]
$$\begin{tikzpicture}[scale=.6]
\draw (0,0) -- (0,6);
\draw (0,6) -- (6,6);
\draw (6,6) -- (6,0);
\draw (6,0) -- (0,0);
\draw [line width=3](0,6) -- (0,4);
\draw [line width=3](0,4) -- (2,4);
\draw [line width=3](2,4) -- (2,3);
\draw [line width=3](2,3) -- (3,3);
\draw [line width=3](3,3) -- (3,2);
\draw [line width=3](3,2) -- (4,2);
\draw [line width=3](4,2) -- (4,0);
\draw [line width=3](4,0) -- (6,0);

\draw [line width=3](0,6) -- (6,6);
\draw [line width=3](6,6) -- (6,0);

\draw [dotted] (0,6) -- (6,0);

\node at (.5,5.4) {{\LARGE *}};
\node at (1.5,5.4) {{\LARGE *}};
\node at (2.5,5.4) {{\LARGE *}};
\node at (3.5,5.4) {{\LARGE *}};
\node at (4.5,5.4) {{\LARGE *}};
\node at (5.5,5.4) {{\LARGE *}};
\node at (.5,4.4) {{\LARGE *}};
\node at (1.5,4.4) {{\LARGE *}};
\node at (2.5,4.4) {{\LARGE *}};
\node at (3.5,4.4) {{\LARGE *}};
\node at (4.5,4.4) {{\LARGE *}};
\node at (5.5,4.4) {{\LARGE *}};
\node at (2.5,3.4) {{\LARGE *}};
\node at (3.5,3.4) {{\LARGE *}};
\node at (4.5,3.4) {{\LARGE *}};
\node at (5.5,3.4) {{\LARGE *}};
\node at (3.5,2.4) {{\LARGE *}};
\node at (4.5,2.4) {{\LARGE *}};
\node at (5.5,2.4) {{\LARGE *}};
\node at (4.5,1.4) {{\LARGE *}};
\node at (5.5,1.4) {{\LARGE *}};
\node at (4.5,0.4) {{\LARGE *}};
\node at (5.5,0.4) {{\LARGE *}};

\node[label=above:{6}] at (3,6) {};
\node[label=left:{2}] at (0,5) {};
\node[label=left:{1}] at (2,3.5) {};
\node[label=left:{1}] at (3,2.5) {};
\node[label=left:{2}] at (4,1) {};

\draw[line width=1.5] (0.5,5.55)--(5.5,5.55)--(5.5,0.55)--(4.5,0.55)--(4.5,4.55)--(0.5,4.55)--(0.5,5.55);
\draw[line width=1.5] (2.5,3.55)--(3.5,3.55)--(3.5,2.55);

\node at (3,-1){$\mathfrak{g}$};

\node at (7,3){\Large$=$};
\end{tikzpicture}
\begin{tikzpicture}[scale=.6]
\draw (0,0) -- (0,6);
\draw (0,6) -- (6,6);
\draw (6,6) -- (6,0);
\draw (6,0) -- (0,0);
\draw [line width=3](0,6) -- (0,4);
\draw [line width=3](0,4) -- (2,4);
\draw [line width=3](2,4) -- (2,3);
\draw [line width=3](2,3) -- (3,3);
\draw [line width=3](3,3) -- (3,2);
\draw [line width=3](3,2) -- (4,2);
\draw [line width=3](4,2) -- (4,0);
\draw [line width=3](4,0) -- (6,0);

\draw [line width=3](0,6) -- (6,6);
\draw [line width=3](6,6) -- (6,0);

\draw [dotted] (0,6) -- (6,0);

\node at (.5,5.4) {{\LARGE *}};
\node at (1.5,5.4) {{\LARGE *}};
\node at (2.5,5.4) {{\LARGE *}};
\node at (3.5,5.4) {{\LARGE *}};
\node at (4.5,5.4) {{\LARGE *}};
\node at (5.5,5.4) {{\LARGE *}};
\node at (.5,4.4) {{\LARGE *}};
\node at (1.5,4.4) {{\LARGE *}};
\node at (2.5,4.4) {{\LARGE *}};
\node at (3.5,4.4) {{\LARGE *}};
\node at (4.5,4.4) {{\LARGE *}};
\node at (5.5,4.4) {{\LARGE *}};
%\node at (2.5,3.4) {{\LARGE *}};
%\node at (3.5,3.4) {{\LARGE *}};
\node at (4.5,3.4) {{\LARGE *}};
\node at (5.5,3.4) {{\LARGE *}};
%\node at (3.5,2.4) {{\LARGE *}};
\node at (4.5,2.4) {{\LARGE *}};
\node at (5.5,2.4) {{\LARGE *}};
\node at (4.5,1.4) {{\LARGE *}};
\node at (5.5,1.4) {{\LARGE *}};
\node at (4.5,0.4) {{\LARGE *}};
\node at (5.5,0.4) {{\LARGE *}};

\draw[line width=1.5] (0.5,5.55)--(5.5,5.55)--(5.5,0.55)--(4.5,0.55)--(4.5,4.55)--(0.5,4.55)--(0.5,5.55);

%\node[label=above:{6}] at (3,6) {};
%\node[label=left:{2}] at (0,5) {};
%\node[label=left:{1}] at (2,3.5) {};
%\node[label=left:{1}] at (3,2.5) {};
%\node[label=left:{2}] at (4,1) {};

%\draw (0.5,5.55)--(5.5,5.55)--(5.5,0.55)--(4.5,0.55)--(4.5,4.55)--(0.5,4.55)--(0.5,5.55);
%\draw (2.5,3.55)--(3.5,3.55)--(3.5,2.55);

\node at (3,-1){$\mathfrak{g}|_2$};

\node at (7,3){$\bigoplus$};
\end{tikzpicture}
\begin{tikzpicture}[scale=.6]
\draw (0,0) -- (0,6);
\draw (0,6) -- (6,6);
\draw (6,6) -- (6,0);
\draw (6,0) -- (0,0);
\draw [line width=3](0,6) -- (0,4);
\draw [line width=3](0,4) -- (2,4);
\draw [line width=3](2,4) -- (2,3);
\draw [line width=3](2,3) -- (3,3);
\draw [line width=3](3,3) -- (3,2);
\draw [line width=3](3,2) -- (4,2);
\draw [line width=3](4,2) -- (4,0);
\draw [line width=3](4,0) -- (6,0);

\draw [line width=3](0,6) -- (6,6);
\draw [line width=3](6,6) -- (6,0);

\draw [dotted] (0,6) -- (6,0);

\draw[line width=1.5] (2.5,3.55)--(3.5,3.55)--(3.5,2.55);

%\node at (.5,5.4) {{\LARGE *}};
%\node at (1.5,5.4) {{\LARGE *}};
%\node at (2.5,5.4) {{\LARGE *}};
%\node at (3.5,5.4) {{\LARGE *}};
%\node at (4.5,5.4) {{\LARGE *}};
%\node at (5.5,5.4) {{\LARGE *}};
%\node at (.5,4.4) {{\LARGE *}};
%\node at (1.5,4.4) {{\LARGE *}};
%\node at (2.5,4.4) {{\LARGE *}};
%\node at (3.5,4.4) {{\LARGE *}};
%\node at (4.5,4.4) {{\LARGE *}};
%\node at (5.5,4.4) {{\LARGE *}};
\node at (2.5,3.4) {{\LARGE *}};
\node at (3.5,3.4) {{\LARGE *}};
%\node at (4.5,3.4) {{\LARGE *}};
%\node at (5.5,3.4) {{\LARGE *}};
\node at (3.5,2.4) {{\LARGE *}};
%\node at (4.5,2.4) {{\LARGE *}};
%\node at (5.5,2.4) {{\LARGE *}};
%\node at (4.5,1.4) {{\LARGE *}};
%\node at (5.5,1.4) {{\LARGE *}};
%\node at (4.5,0.4) {{\LARGE *}};
%\node at (5.5,0.4) {{\LARGE *}};

%\node[label=above:{6}] at (3,6) {};
%\node[label=left:{2}] at (0,5) {};
%\node[label=left:{1}] at (2,3.5) {};
%\node[label=left:{1}] at (3,2.5) {};
%\node[label=left:{2}] at (4,1) {};

%\draw (0.5,5.55)--(5.5,5.55)--(5.5,0.55)--(4.5,0.55)--(4.5,4.55)--(0.5,4.55)--(0.5,5.55);
%\draw (2.5,3.55)--(3.5,3.55)--(3.5,2.55);

\node at (3,-1){$\mathfrak{g}|_1$};
\end{tikzpicture}$$
    \caption{Vector space decomposition of $\mathfrak{g}$}
    \label{fig:decomp}
\end{figure}
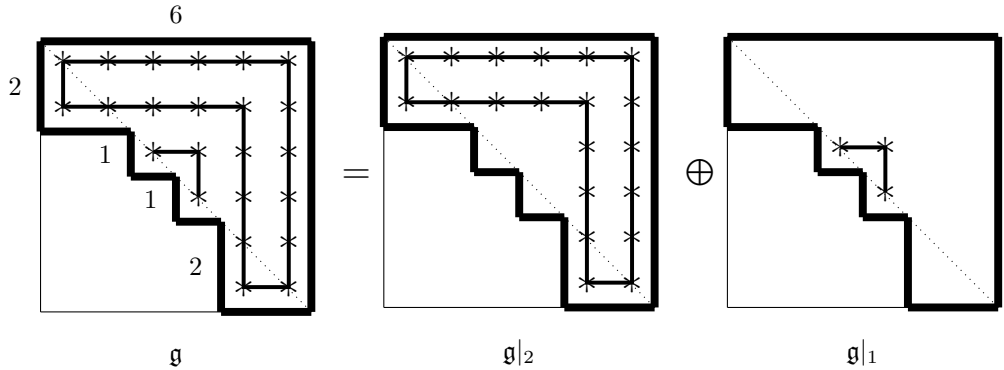
\end{Ex}

%\noindent Now, orient $CM(\mathfrak{g})$ counter-clockwise (i.e., top edges are directed from right to left, bottom edges are directed left to right). Let $E_{CM(\mathfrak{g})}$ be the set of edges in the \textit{directed component meander}.

We are now ready to formally define the core and peak of $\mathfrak{g}$.  First, fix a component $\g|_{c_i}$ of homotopy type $c_i$ and define the sets

$$
V_{c_i}=\{I\;|\;v_I\text{ is a vertex on the path of $CM(\g)$ corresponding to }c_i\},$$

$$\textgoth{C}_{c_i}=\{I\times I\;|\;I\in V_{c_i}\}.$$

$$\textgoth{P}_{c_i}=\{I\times J\;|\;I,J\in V_{c_i}\text{ and } (v_I,v_J) \text{ is an edge in }\overrightarrow{CM}(\mathfrak{g})\}.$$
The set $\textgoth{C}_{c_i}$ is the \textit{core set of $\g|_{c_i}$} -- the set of $c_i\times c_i$ blocks on the diagonal of $\g$ contained in $\g|_{c_i}$ -- and the set $\textgoth{P}_{c_i}$ is the \textit{peak set of $\mathfrak{g}|_{c_i}.$}  Now, we define the \textit{core of $\g$} and the \textit{peak of $\g$} as the union of the core sets and peak sets, respectively, of the components in the vector space decomposition (\ref{eqn:decomp}). In other words,

$$\textgoth{C}_\g=\bigcup_{i=1}^q\textgoth{C}_{c_i}\hspace{2em}\text{ and }\hspace{2em}\textgoth{P}_\g=\bigcup_{i=1}^q\textgoth{P}_{c_i}.$$

\noindent 
Since $\textgoth{C}_{\mathfrak{g}}$ and $\textgoth{P}_{\mathfrak{g}}$ consist of sets of ordered pairs of indices defining blocks of admissible locations in $\mathfrak{g},$ we refer to elements of $\textgoth{C}_{\g}$ and $\textgoth{P}_{\g}$ as \textit{core blocks} and \textit{peak blocks}, respectively.

\begin{Ex}
We illustrate the definitions above by constructing the core and peak sets of the seaweed of Example~\ref{ex:componentmeander}, $\g=\mathfrak{p}_6^A\frac{2|1|1|2}{6}.$ Recall $\mathfrak{g}$ has homotopy type $\mathcal{H}(2,1).$ As illustrated in Figure~\ref{fig:commeanex} \textup(center\textup) above, the vertices of $CM(\g)$ are written as $v_{\{1,2\}},v_{\{3\}},v_{\{4\}},$ and $v_{\{5,6\}}.$ So, we have that $$V_2=\Big\{\{1,2\},\{5,6\}\Big\}\quad\quad\text{and}\quad\quad V_1=\Big\{\{3\},\{4\}\Big\},$$
so the core set of $\mathfrak{g}$ is $$\textgoth{C}_{\g}=\Big\{\{(1,1),(1,2),(2,1),(2,2)\},\{(3,3)\},\{(4,4)\},\{(5,5),(5,6),(6,5),(6,6)\}\Big\}.$$

\noindent Now, to construct the peak set of $\g,$ consider $\overrightarrow{CM}(\mathfrak{g}).$ As illustrated by Figure~\ref{fig:commeanex} \textup(right\textup), we have that $$\textgoth{P}_2=\Big\{\{(1,5),(1,6),(2,5),(2,6)\}\Big\}\quad\quad\text{and}\quad\quad\textgoth{P}_1=\Big\{\{(3,4)\}\Big\},$$ so the peak set of $\mathfrak{g}$ is $$\textgoth{P}_{\mathfrak{g}}=\Big\{\{(1,5),(1,6),(2,5),(2,6)\},\{(3,4)\}\Big\}.$$

\noindent 
The core blocks and peak blocks of $\g$ are bolded and outlined, respectively, in the seaweed in Figure~\ref{fig:cp}.

\begin{figure}[H]
$$\begin{tikzpicture}[scale=.6]
\draw (0,0) -- (0,6);
\draw (0,6) -- (6,6);
\draw (6,6) -- (6,0);
\draw (6,0) -- (0,0);
\draw [line width=3](0,6) -- (0,4);
\draw [line width=3](0,4) -- (2,4);
\draw [line width=3](2,4) -- (2,3);
\draw [line width=3](2,3) -- (3,3);
\draw [line width=3](3,3) -- (3,2);
\draw [line width=3](3,2) -- (4,2);
\draw [line width=3](4,2) -- (4,0);
\draw [line width=3](4,0) -- (6,0);

\draw [line width=3](0,6) -- (6,6);
\draw [line width=3](6,6) -- (6,0);

\draw [dotted] (0,6) -- (6,0);

\node at (.5,5.4) {{\LARGE \bf*}};
\node at (1.5,5.4) {{\LARGE \bf*}};
\node at (2.5,5.4) {{\LARGE *}};
\node at (3.5,5.4) {{\LARGE *}};
\node at (4.5,5.4) {{\LARGE *}};
\node at (5.5,5.4) {{\LARGE *}};
\node at (.5,4.4) {{\LARGE \bf*}};
\node at (1.5,4.4) {{\LARGE \bf*}};
\node at (2.5,4.4) {{\LARGE *}};
\node at (3.5,4.4) {{\LARGE *}};
\node at (4.5,4.4) {{\LARGE *}};
\node at (5.5,4.4) {{\LARGE *}};
\node at (2.5,3.4) {{\LARGE \bf*}};
\node at (3.5,3.4) {{\LARGE *}};
\node at (4.5,3.4) {{\LARGE *}};
\node at (5.5,3.4) {{\LARGE *}};
\node at (3.5,2.4) {{\LARGE \bf*}};
\node at (4.5,2.4) {{\LARGE *}};
\node at (5.5,2.4) {{\LARGE *}};
\node at (4.5,1.4) {{\LARGE \bf*}};
\node at (5.5,1.4) {{\LARGE \bf*}};
\node at (4.5,0.4) {{\LARGE \bf*}};
\node at (5.5,0.4) {{\LARGE \bf*}};

\node[label=above:{6}] at (3,6) {};
\node[label=left:{2}] at (0,5) {};
\node[label=left:{1}] at (2,3.5) {};
\node[label=left:{1}] at (3,2.5) {};
\node[label=left:{2}] at (4,1) {};

\draw (4.1,6)--(4.1,4.1)--(6,4.1);
\draw (3.1,3.1)--(3.1,3.9)--(3.9,3.9)--(3.9,3.1)--(3.1,3.1);

%\draw (0.5,5.55)--(5.5,5.55)--(5.5,0.55)--(4.5,0.55)--(4.5,4.55)--(0.5,4.55)--(0.5,5.55);
%\draw (2.5,3.55)--(3.5,3.55)--(3.5,2.55);
\end{tikzpicture}$$
\caption{The core and peak blocks identified within $\g$}
\label{fig:cp}
\end{figure}
\end{Ex}

\subsection{Regular one-forms}

The core and peak of an index-one $\mathfrak{g}$ facilitate the definition of a family $\Phi$ of regular one-forms on $\g$ reliant on the homotopy type of $\mathfrak{g}$. 
An element of $\Phi$ is of the form $\sum e_{i,j}^*$, where $e_{i,j}^*$ denotes the dual of the matrix unit $e_{i,j}$ and the sum is over a restricted set of admissible locations $(i,j)$ of $\mathfrak{g}$.  The restrictions are determined by the homotopy type of $\mathfrak{g}$, which, for an index-one $\g$, must be either $\mathcal{H}(2)$ or $\mathcal{H}(1,1)$.

%particular regular one-forms $\overline{F}$ for a general seaweed -- regardless of index or homotopy type. The construction, as well as the definitions given above, appear in \textbf{\cite{Adiss}}. However, we do not require the full power of the construction in that reference. In particular, the construction of $\overline{F}$ involves creating regular one-forms for $\mathfrak{gl}(n),$ $\mathfrak{sl}(n)$ -- which are seaweeds themselves -- and embedding appropriately into the core of the seaweed. Since the only seaweeds of interest here are index one, we shall use a more user-friendly description whenever possible. Ongoing, we let $e_{i,j}^*\in\mathfrak{g}^*$ denote the one-form which returns the $(i,j)$ entry of an element of $\mathfrak{g}.$

%For a general type-A seaweed, $\overline{F}=\sum e_{i,j}^*,$ where the $(i,j)$ are determined by the structure of the core and peak blocks. Therefore, the general framework is reliant on the homotopy type of the particular seaweed in question. In the following two sections, we examine the construction of $\overline{F}$ in each of the homotopy types $\mathcal{H}(1,1)$ and $\mathcal{H}(2),$ and we explicitly identify $\ker(B_{\overline{F}})$ in each case.

\subsection{Homotopy type $\mathcal{H}(2)$}\label{sec:h2}

Let $\g$ have homotopy type $\mathcal{H}(2)$, with attendant peak and core.  Here, we construct a regular one-form $\varphi_{(2)}\in\Phi$ 
using $\textgoth{P}_{\mathfrak{g}}$ and $\textgoth{C}_{\mathfrak{g}}$. We find it convenient to mark the admissible locations which define $\varphi_{(2)}$ by black dots -- a black dot is placed in location $(i,j)$ if and only if $e_{i,j}^*$ is a nonzero summand of $\varphi_{(2)}$ -- according to the following schema. Note that the core and peak blocks are all $2\times 2$. A single black dot is placed in the upper left corner of each of the core blocks.  For the peak blocks, black dots are placed along the diagonal.
See Example~\ref{ex:Afunc}.

\begin{Ex}\label{ex:Afunc}
Recall that $\mathfrak{g}=\mathfrak{p}_8^A \frac{2|6}{8}$ has homotopy type $\mathcal{H}(2).$ The regular functional $\varphi_{(2)}$ has summands determined by locations of black dots in Figure~\ref{fig:Adots}.

\begin{figure}[H]
$$\begin{tikzpicture}[scale=0.45]
\def\Node{\node [circle, fill, inner sep=2pt]}
	\draw (0,0) -- (0,8);
	\draw (0,8) -- (8,8);
	\draw (8,8) -- (8,0);
	\draw (8,0) -- (0,0);
	\draw [line width=3](0,8) -- (0,6);
	\draw [line width=3](0,6) -- (2,6);
	\draw [line width=3](2,6) -- (2,0);
	\draw [line width=3](2,0) -- (8,0);
	\draw [line width=3](0,8) -- (8,8);
	\draw [line width=3](8,8) -- (8,0);
	\Node at (0.5,7.4) {};
	\Node at (6.5,7.4) {};
	\Node at (7.5,6.4) {};
	\Node at (2.5,5.4) {};
	\Node at (4.5,3.4) {};
	\Node at (3.5,0.4) {};
	\Node at (2.5,1.4) {};
	\Node at (6.5,1.4) {};
	\Node at (4.5,5.4) {};
	\Node at (5.5,4.4) {};
	\node [label=above:{8}] at (4,8){};
	\node [label=left:{2}] at (0,7) {};
	\node [label=left:{6}] at (2,3) {};
	\draw (6,8)--(6,6)--(8,6);
	\draw (6,6)--(4,6)--(4,4)--(6,4)--(6,6);
	\draw (2,2)--(4,2)--(4,0);
\end{tikzpicture}$$
\caption{The summands of $\varphi_{(2)}$ identified within $\g$}
\label{fig:Adots}
\end{figure}
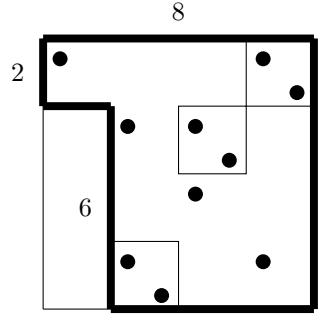
\noindent
As Figure \ref{fig:Adots} displays, $\varphi_{(2)}$ is given by $$\varphi_{(2)}=e_{1,1}^*+e_{3,3}^*+e_{5,5}^*+e_{7,7}^*+e_{1,7}^*+e_{2,8}^*+e_{3,5}^*+e_{4,6}^*+e_{7,3}^*+e_{8,4}^*.$$
\end{Ex}

\subsubsection{The kernel of $\varphi_{(2)}$ }
In the proof of the main theorem in Section~\ref{sec:main}, we require a specific description of $\ker(B_{\varphi_{(2)}})=span\{k\}.$ To identify $k,$ we first require the following technical lemmas.

\begin{lemma}\label{lem:evenpart}
If $\g = \mathfrak{p}_n^A \frac{a_1|\dots|a_m}{b_1|\dots|b_t}$ has homotopy type $\mathcal{H}(2)$, then $a_i$ and $b_j$ are even for all $1\leq i\leq m$ and $1\leq j\leq t$.
\end{lemma}
\begin{proof}
Since $\mathfrak{g}$ has homotopy type $\mathcal{H}(2),$ its meander $M(\mathfrak{g})$ consists of exactly one cycle; moreover, each vertex has degree 2. In particular, each vertex is incident with exactly one top edge and exactly one bottom edge. However, if $a_i$ (resp. $b_j$) is odd for some $i$ (resp. $j$), then the middle vertex of block $a_i$ (resp. $b_j$), i.e., $v_{\sum_{k=1}^{i-1}a_k+\big\lceil \frac{a_i}{2}\big\rceil}$ \Big(resp. $v_{\sum_{k=1}^{j-1}b_k+ \big\lceil\frac{b_j}{2}\big\rceil}$\Big) is not incident to a top (resp. bottom) edge. The result follows.
\end{proof}

\begin{lemma}\label{lem:oppsign}
Let $\mathfrak{g}$ have homotopy type $\mathcal{H}(2).$ If $e^*_{i,j}$ occurs as a nontrivial summand in $\varphi_{(2)}$, then $i$ and $j$ have the same parity.
\end{lemma}
\begin{proof}
Let $\mathfrak{g}=\mathfrak{p}_n^A\frac{a_1|\dots|a_m}{b_1|\dots|b_t}$ have homotopy type $\mathcal{H}(2),$ and without loss of generality, assume there exists $s$ such that $a_s>2.$ Consider the collection of peak blocks determined by $a_s;$ in particular, consider the summands of $\varphi_{(2)}$ within each of these peak blocks. Such summands are
$e_{i,j}^*$ and $e_{i+1,j+1}^*,$ where $$i=\sum_{r=1}^{s}a_{r}-1+2t\text{\ \ \ \ \  and\ \ \ \ \ } j=\sum_{r=1}^{s-1}a_{r}+1+2t$$
%$e_{r,s}^*,$ where $$r=\sum_{\ell=1}^{k-1}a_{\ell}+2-2t\text{\ \ \ \ \ and\ \ \ \ \ }s=\sum_{\ell=1}^ka_{\ell}-2t,$$ 
for $0\le t\le \lfloor\frac{a_s}{4}\rfloor$. Notice that $i$ and $j$ are both odd for all $t$, by Lemma~\ref{lem:evenpart}. Since $a_s$ was arbitrary and the same argument applies for all $b_s,$ the result follows.
\end{proof}

\noindent
With Lemmas~\ref{lem:evenpart} and \ref{lem:oppsign} established, the following result identifies the generator of $\ker(B_{\varphi_{(2)}}).$

\begin{theorem}\label{thm:h2a}
Let $\mathfrak{g}\subset\mathfrak{sl}(n)$ be a type-A seaweed with homotopy type $\mathcal{H}(2)$. If $\varphi_{(2)}$ is defined as above, then $$\ker(B_{\varphi_{(2)}})=\textup{span}\{k\},$$ where $$k=\sum_{i=1}^n(-1)^{i+1}e_{i,i}.$$
\end{theorem}
\begin{proof}
Since $\varphi_{(2)}$ is regular on $\mathfrak{g},$ we need only show that $k\in\ker(B_{\varphi_{(2)}}).$ As a consequence of Lemma~\ref{lem:evenpart}, we have that $n$ is even, so $tr(k)=0$ and $k\in\mathfrak{g}.$ Now, to see that $k\in\ker(B_{\varphi_{(2)}}),$ consider the following:
\begin{itemize}
    \item[\textbf{(a)}] $\varphi_{(2)}([k,e_{i,i}-e_{i+1,i+1}])=0,$ for all $1\leq i\leq n-1,$
    \item[\textbf{(b)}] $\varphi_{(2)}([k,e_{i,j}])=0,$ for all $e_{i,j}\in\mathfrak{g}$ such that $e_{i,j}^*$ is not a summand of $\varphi_{(2)},$ and
    \item[\textbf{(c)}] $\varphi_{(2)}([k,e_{i,j}])=\varphi_{(2)}(e_{i,j}-e_{i,j})=0,$ for all $e_{i,j}\in\mathfrak{g}$ such that $e_{i,j}^*$ is a summand of $\varphi_{(2)}.$
\end{itemize}
Equations \textbf{(a)} and \textbf{(b)} follow immediately from $k$ being a diagonal element of $\mathfrak{g},$ and Equation \textbf{(c)} follows from Lemma~\ref{lem:oppsign}. Therefore, $k\in\ker(B_{\varphi_{(2)}}).$
\end{proof}

\begin{Ex}
Returning to the seaweed $\mathfrak{g}=\mathfrak{p}_8^A\frac{2|6}{8}$ from Example~\ref{ex:Afunc}, we have that $$\ker(B_{\varphi_{(2)}})=\textup{span}\{e_{1,1}-e_{2,2}+e_{3,3}-e_{4,4}+e_{5,5}-e_{6,6}+e_{7,7}-e_{8,8}\}.$$
\end{Ex}

\subsection{Homotopy type $\mathcal{H}(1,1)$}\label{sec:h11}

\textit{Notation}: Throughout this section, for a given graph $G,$ we denote its vertex set by $V(G)$ and its edge set by $E(G).$

\medskip
When $\g$ has homotopy type $\mathcal{H}(1,1),$ 
the core and peak blocks have dimension one, so $CM(\g)=M(\g)$.  Moreover, 
$M(\g)$ must consist of two paths, say 
$P_1$ and $P_2$. Let $P_1$ be the path which contains $v_1$.  Now define the one-form $\varphi_{(1,1)}\in \Phi$ as follows:

\begin{eqnarray}\label{summation}
\varphi_{(1,1)}=\sum_{(v_i,v_j)\in E(\overrightarrow{M}(\g))}e_{i,j}^*+\sum_{v_i\in V(P_1)}e_{i,i}^*.
\end{eqnarray}

\noindent
Note that all the terms in the first summation of (\ref{summation}) can be read directly from the edges of the directed meander in the obvious way. See Example \ref{ex:h11form}.

%\footnote{The choice of $F_{1,1}$ for a seaweed with homotopy type $\mathcal{H}(1,1)$ corresponds to -- using the terminology and notation of \textbf{\cite{Adiss}} -- embedding the one-form $K_1$ into the core of one component and embedding the one-form $F_1$ into the core of the other component.} 

\begin{Ex}\label{ex:h11form}
Consider the seaweed $\mathfrak{g}=\mathfrak{p}_5^A\frac{1|4}{3|1|1}$ with directed meander shown in Figure~\ref{fig:H11ex} below. Note that $\mathfrak{g}$ has homotopy type $\mathcal{H}(1,1),$ so the above construction yields the regular one-form $$\varphi_{(1,1)}=e_{1,3}^*+e_{4,3}^*+e_{5,2}^*+e_{1,1}^*+e_{3,3}^*+e_{4,4}^*.$$

\begin{figure}[H]
$$\begin{tikzpicture}
\def\Node{\node [circle, fill, inner sep=1.5pt]}
\tikzset{->-/.style={decoration={
  markings,
  mark=at position .55 with {\arrow{>}}},postaction={decorate}}}
  \Node[label=below:\footnotesize {$v_1$}] at (0,0) {};
  \Node[label=below:\footnotesize {$v_2$}] at (1,0) {};
  \Node[label=below:\footnotesize {$v_3$}] at (2,0) {};
  \Node[label=below:\footnotesize {$v_4$}] at (3,0) {};
  \Node[label=below:\footnotesize {$v_5$}] at (4,0) {};
  \draw[->-] (0,0) to[bend right=60] (2,0);
  \draw[->-] (4,0) to[bend right=60] (1,0);
  \draw[->-] (3,0) to[bend right=60] (2,0);
\end{tikzpicture}$$
    \caption{\centering The directed meander $\protect\overrightarrow{M}(\mathfrak{g})$}
    \label{fig:H11ex}
\end{figure}
\end{Ex}

\subsubsection{The kernel of $\varphi_{(1,1)}$ }

\noindent By construction, the one-form $\varphi_{(1,1)}$ is regular. As in the $\mathcal{H}(2)$ case, we can also describe the kernel of its associated Kirillov form. 
See Theorem \ref{thm:h11ker}.

\begin{theorem}\label{thm:h11ker}
If $\mathfrak{g}$ has homotopy type $\mathcal{H}(1,1)$ with $\varphi_{(1,1)}$  defined as in \textup{(\ref{summation})}, then $$\ker(B_{\varphi_{(1,1)}})=\textup{span}\{h\},$$ where $$h=|V(P_2)|\sum_{i\in V(P_1)}e_{i,i}-|V(P_1)|\sum_{j\in V(P_2)}e_{j,j}.$$
\end{theorem}
\begin{proof}
Since $\varphi_{(1,1)}$ is regular on $\mathfrak{g},$ we need only show that $h\in\ker(B_{\varphi_{(1,1)}}).$  We first establish that $h\in\mathfrak{g}$ by noting that $tr(h)=|V(P_2)||V(P_1)|-|V(P_1)||V(P_2)|=0$. Now, to see that $h\in\ker(B_{\varphi_{(1,1)}}),$ consider the following:
\begin{itemize}
    \item $\varphi_{(1,1)}([h,e_{i,i}-e_{i+1,i+1}])=0,$ for all $1\leq i\leq |V(P_1)|+|V(P_2)|-1,$
    \item $\varphi_{(1,1)}([h,e_{i,j}])=0,$ for all $e_{i,j}\in\mathfrak{g}$ such that $e_{i,j}^*$ is not a summand of $\varphi_{(1,1)},$
    \item $\varphi_{(1,1)}([h,e_{i,j}])=\varphi_{(1,1)}(|V(P_2)|(e_{i,j}-e_{i,j}))=0,$ for all $e_{i,j}\in\mathfrak{g}$ such that $(v_i,v_j)\in E(P_1),$ and
    \item $\varphi_{(1,1)}([h,e_{i,j}])=\varphi_{(1,1)}(|V(P_1)|(e_{i,j}-e_{i,j}))=0,$ for all $e_{i,j}\in\mathfrak{g}$ such that $(v_i,v_j)\in E(P_2).$
\end{itemize}
The result follows.
\end{proof}

\begin{Ex}
Returning to the seaweed $\mathfrak{g}=\mathfrak{p}_5^A\frac{1|4}{3|1|1}$ from Example~\ref{ex:h11form}, we have that $$\ker(B_{\varphi_{(1,1)}})=\textup{span}\{2e_{1,1}-3e_{2,2}+2e_{3,3}+2e_{4,4}-3e_{5,5}\}.$$
\end{Ex}

\section{Main results}\label{sec:main}
In this section, we establish the main result of this article. We first need some notation and a few general results about contact Lie algebras.

\medskip
\noindent
\textit{Notation:  In this section, we will use 
$\mathfrak{f}$ to denote an  arbitrary Lie algebra and we will continue with our established convention of letting $\mathfrak{g}$ represent a type-$A$ seaweed.}

\medskip

Let $\mathfrak{f}$ be a $(2k+1)-$dimensional contact Lie algebra with contact form $\varphi.$ Fix an ordered basis\linebreak $\mathscr{B}=\{E_1,\dots,E_{2k+1}\}$ of $\mathfrak{f}$ and define $C(\mathfrak{f},\mathscr{B})=([E_i,E_j])_{1\leq i,j\leq 2k+1}$ to be the \textit{commutator matrix} associated to $\mathfrak{f}$ and indexed by basis $\mathscr{B}$. The contact form $\varphi$ can be applied to each element of $C(\mathfrak{f},\mathscr{B})$ to yield the matrix
%($\mathscr{B}$ is omitted if chosen basis is clear).
  $$[B_{\varphi}]=\varphi(C(\mathfrak{f},\mathscr{B})).$$  Denote by $[\varphi]=(x_1\dots x_{2k+1})^t$ the coordinate vector in $\mathbb{C}^{2k+1}$ such that $\varphi=\sum_{i=1}^{2k+1} x_iE_i^*$ where $\{E_1^*,\dots,E_{2k+1}^*\}$ is the ``dual basis'' associated to $\mathscr{B}$ and consider the square 
   $(2k+2)-$dimensional skew-symmetric matrix $$\left[\widehat{B}_{\varphi}\right]=\begin{bmatrix}
0 & [\varphi]^t\\
-[\varphi] & \varphi(C(\mathfrak{\mathfrak{f}}))
\end{bmatrix}.$$    

\noindent
A straightforward computation yields the following technical lemma.

\begin{lemma}[Salgado \textbf{\cite{Sally}}, 2019]\label{lem:det}
Let $\mathfrak{f}$ be a Lie algebra with $\dim\mathfrak{g}=2k+1$ and $\varphi\in\mathfrak{f}^*.$ Using the notation developed above, $$\varphi\wedge(d\varphi)^k=\det\left(\left[\widehat{B}_{\varphi}\right]\right) E_1^*\wedge \dots \wedge E_{2k+1}^*. $$ Therefore, $\varphi$ is a contact form on $\mathfrak{f}$ if and only if $$\det\left(\left[\widehat{B}_{\varphi}\right]\right)\neq 0.$$
\end{lemma}

\noindent From Lemma~\ref{lem:det} above, we obtain the following useful characterization of a contact Lie algebra.

\begin{lemma}\label{lem:kernel}
An index-one Lie algebra $\mathfrak{f}$ is contact if and only if there exists a regular one-form $\varphi\in\mathfrak{f}^*$ for which there is an element $x\in\mathfrak{f}$ with $\ker(B_{\varphi})=span\{ x\}$ and $\varphi(x)\neq 0.$
\end{lemma}
\begin{proof}
To establish the forward implication, we need only recall that every contact form $\varphi$ has a unique Reeb vector, $x_{\varphi}\in\mathfrak{f}$ defined by the equations $$B_{\varphi}(x_{\varphi},-)=0$$ and $$\varphi(x_{\varphi})=1.$$ 

For the reverse implication, let $\varphi\in\mathfrak{f}^*$ be a regular one-form, i.e., $\dim\ker(B_{\varphi})=\ind\mathfrak{f}=1.$ Let $x\in\mathfrak{f}$ denote a generator of $\ker(B_{\varphi}),$ then extend $x$ to a basis of $\mathfrak{f},$ and consider $\det\left[\widehat{B}_{\varphi}\right]$ with respect to this basis. Computing, we have that $$\det\left[\widehat{B}_{\varphi}\right]=\varphi(x)^2\det\left[B'_{\varphi}\right],$$ where $B'_{\varphi}$ is the submatrix of $\left[B_{\varphi}\right]$ with the row and column indexed by $\varphi(x)$ removed. Since $\varphi$ is a regular one-form, $B'_{\varphi}$ has full rank, so $\det\left[B'_{\varphi}\right]\neq 0.$ The result follows from Lemma~\ref{lem:det}.
\end{proof}

We are now in a position to prove the main theorem of this article.

\begin{theorem}\label{thm:main}
If $\mathfrak{g}$ is a type-A seaweed, then $\mathfrak{g}$ is contact if and only if $\ind\mathfrak{g}=1.$
\end{theorem}
\begin{proof}

Let $\mathfrak{g}$ be an index-one seaweed subalgebra of $\mathfrak{sl}(n)$. Recall from Theorem~\ref{cor:index A} that this means $\mathfrak{g}$ has one of the following homotopy types: $\mathcal{H}(2)$ or $\mathcal{H}(1,1)$. We proceed by treating each homotopy type as its own case, using the constructions and notation of Section~\ref{sec:framework} and then applying Lemma~\ref{lem:kernel}.

\medskip
\noindent
\textbf{Case 1:} $\mathfrak{g}$ has homotopy type $\mathcal{H}(2).$ Using the notation of Section~\ref{sec:h2}, we claim that $\varphi_{(2)}$ is a contact form on $\mathfrak{g}.$ Recall from Theorem~\ref{thm:h2a} that $$\ker(B_{\varphi_{(2)}})=\textup{span}\{k\}=\textup{span}\left\{\sum_{i=1}^n(-1)^{i+1}e_{i,i}\right\},$$ and notice that $\varphi_{(2)}(k)=\frac{n}{2}\neq 0.$ An application of Lemma~\ref{lem:kernel} establishes the claim.

\medskip
\noindent
\textbf{Case 2:} $\mathfrak{g}$ has homotopy type $\mathcal{H}(1,1).$ Using the notation of Section~\ref{sec:h11}, we claim that $\varphi_{(1,1)}$ is a contact form on $\mathfrak{g}.$ Recall from Theorem~\ref{thm:h11ker} that $$\ker(B_{\varphi_{(1,1)}})=\textup{span}\{h\}=\textup{span}\left\{|V(P_2)|\sum_{i\in V(P_1)}e_{i,i}-|V(P_1)|\sum_{j\in V(P_2)}e_{j,j}\right\},$$ and notice that $\varphi_{(1,1)}(h)=|V(P_2)||V(P_1)|\neq 0.$ An application of Lemma~\ref{lem:kernel} establishes the claim.
\end{proof}

The following corollary identifies the only $n$ for which $\mathfrak{sl}(n,\mathbb{C})$ is contact, providing another proof for a well-known result (cf. \textbf{\cite{RG}}).

\begin{corollary}
The Lie algebra $\mathfrak{sl}(n,\mathbb{C})$ is contact if and only if $n=2.$
\end{corollary}
\begin{proof}
Note that $\mathfrak{sl}(n,\mathbb{C})$ is a type-$A$ seaweed; in particular, it is the seaweed $\mathfrak{p}_n^A\frac{n}{n}.$ Further, its meander consists of $\lfloor\frac{n}{2}\rfloor$ cycles and, if $n$ is odd, a degenerate path. By Theorem 1, we have that $\ind\mathfrak{sl}(n,\mathbb{C})=1$ if and only if $n=2,$ and the result follows from an application of Theorem~\ref{thm:main}.
\end{proof}

\section{Examples}\label{sec:examples}
Theorem~\ref{thm:indexA} provides an elegant combinatorial formula for the index -- and is critical to the proof heuristics of our main Theorem~\ref{thm:main} -- but to use it to quickly identify, or construct, contact seaweeds is nettlesome, having to first construct the associated meander and then count the number and type of connected components.  Fortunately, for seaweeds consisting of a small number of parts, one can determine the index with some dispatch. In particular, the following theorem provides an explicit index formula presented in terms of a linear common divisor of two arguments, each of which is a linear combination of the terms in the seaweed's defining composition (see \textbf{\cite{Cam}}).  We can use these formulas to manufacture an unlimited supply of contact seaweeds, of arbitrarily large dimension. 

\begin{theorem}[Coll et al. \textbf{\cite{meanders2}}, 2015]\label{thm:3parts}
The seaweed $\mathfrak{p}_n^A\frac{a|b|c}{n}$ has index $\gcd(a+b,b+c)-1.$
\end{theorem}

\noindent
So, to generate contact seaweeds of the form $\mathfrak{p}_n^A\frac{a|b|c}{n},$ we need only ensure $\gcd(a+b,b+c)=2.$  This is easy to do; consider, for example, $\mathfrak{p}_5^A\frac{1|1|3}{5}$, $\mathfrak{p}_7^A\frac{1|3|3}{7}$, $\mathfrak{p}_{12}^A\frac{4|2|6}{12},$ etc.  Moreover, if $b=0$, an immediate corollary to Theorem~\ref{thm:3parts} gives an index formula in the maximal parabolic case.

\begin{theorem}[Elashvili \textbf{\cite{Elashvili}}, 1990]\label{thm:elashvili}
The seaweed $\mathfrak{p}_n^A\frac{a|c}{n}$ has index $\gcd(a,c)-1.$
\end{theorem}
\noindent
Using Theorem \ref{thm:elashvili}, the reader will have no difficulty constructing examples of maximally parabolic type-$A$ contact seaweeds.
Of course, Theorems~\ref{thm:3parts} and~\ref{thm:elashvili} can also be used to generate Frobenius seaweeds since, by Theorem \ref{thm:indexA}, their associated meanders must consist of a single path. Here are a couple of examples: $\mathfrak{p}_5^A\frac{2|3}{5}$ and $\mathfrak{p}_8^A\frac{1|2|5}{8}.$ 
See Figure~\ref{fig:frob},  where the meanders of these two Frobenius seaweeds are displayed.

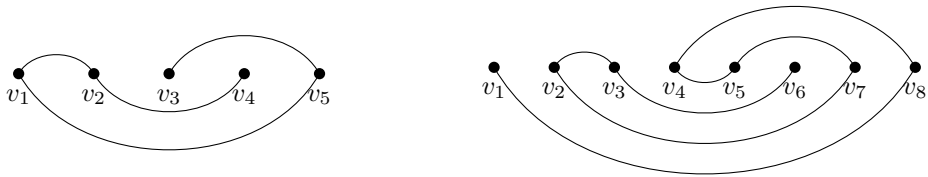
\begin{figure}[H]
$$\begin{tikzpicture}
\def\Node{\node [circle, fill, inner sep=1.5pt]}
\Node[label=below:{$v_1$}] at (0,0){};
\Node[label=below:{$v_2$}] at (1,0){};
\Node[label=below:{$v_3$}] at (2,0){};
\Node[label=below:{$v_4$}] at (3,0){};
\Node[label=below:{$v_5$}] at (4,0){};
\node[color=white] at (0,-1.7){};
\draw (0,0) to[bend right=60] (4,0);
\draw (0,0) to[bend left=60] (1,0);
\draw (1,0) to[bend right=60] (3,0);
\draw (2,0) to[bend left=60] (4,0);
\end{tikzpicture}
\hspace{5em}
\begin{tikzpicture}[scale=0.8]
\def\Node{\node [circle, fill, inner sep=1.5pt]}
\Node[label=below:{$v_1$}] at (0,0){};
\Node[label=below:{$v_2$}] at (1,0){};
\Node[label=below:{$v_3$}] at (2,0){};
\Node[label=below:{$v_4$}] at (3,0){};
\Node[label=below:{$v_5$}] at (4,0){};
\Node[label=below:{$v_6$}] at (5,0){};
\Node[label=below:{$v_7$}] at (6,0){};
\Node[label=below:{$v_8$}] at (7,0){};
\draw (0,0) to[bend right=60] (7,0);
\draw (1,0) to[bend right=60] (6,0);
\draw (1,0) to[bend left=60] (2,0);
\draw (2,0) to[bend right=60] (5,0);
\draw (3,0) to[bend right=60] (4,0);
\draw (3,0) to[bend left=60] (7,0);
\draw (4,0) to[bend left=60] (6,0);
\end{tikzpicture}$$
\caption{Meanders}
\label{fig:frob}
\end{figure}

\end{document}